\documentclass[12pt,twoside,reqno]{amsart}

\usepackage[OT1]{fontenc}    %
\usepackage{type1cm}         %
\usepackage{amsthm}          %
\usepackage[dvips]{graphicx} % for gfx
\usepackage[bf]{caption}     % captionstyles
\usepackage{psfrag}          % replace text in figures
\usepackage{amsfonts}
\usepackage{amsmath}
\usepackage{amssymb}
\usepackage{wasysym}

\usepackage{verbatim}        % comment the following line to get rid
\numberwithin{equation}{section}

\theoremstyle{plain}
\newtheorem{theorem}{Theorem}[section]

\newtheorem{lemma}[theorem]{Lemma}

\theoremstyle{remark}

\theoremstyle{definition}

\newcommand{\R}{\mathbb{R}}

\newcommand{\Z}{\mathbb{Z}}
\newcommand{\N}{\mathbb{N}}

\newcommand{\roo}{\varrho}

\DeclareMathOperator{\dist}{dist}

\DeclareMathOperator{\por}{por}

\begin{document}

\title{Porosity and regularity in metric measure spaces}

\author{Antti K\"aenm\"aki}

\address{Department of Mathematics and Statistics \\
         P.O. Box 35 (MaD) \\
         FI-40014 University of Jyv\"askyl\"a \\
         Finland}

\email{antakae@maths.jyu.fi}

\thanks{The author acknowledges the support of the Academy of Finland
  (project \#114821).}
\date{\today}

\begin{abstract}
  This is a report of a joint work with E.\ J\"arvenp\"a\"a, M.\
  J\"arvenp\"a\"a, T.\ Rajala, S.\ Rogovin, and V.\ Suomala. In
  \cite{our_paper}, we characterized uniformly porous sets in
  $s$-regular metric spaces in terms of regular sets by verifying that
  a set $A$ is uniformly porous if and only if there is $t < s$ and a
  $t$-regular set $F \supset A$. Here we outline the main idea of the
  proof and also present an alternative proof
  for the crucial lemma needed in the proof of the result.
\end{abstract}

\maketitle

\section{Setting and result}

We assume that $X = (X,d)$ is a complete separable metric space.
If $s > 0$ then a measure $\mu$ on $X$ is called \emph{$s$-regular on
  a set $A \subset X$} provided that there are constants $0 <a_\mu \le
b_\mu < \infty$ and $r_\mu > 0$ such that
\begin{equation*}
  a_\mu r^s \le \mu\bigl( B(x,r) \bigr) \le b_\mu r^s
\end{equation*}
for every $x \in A$ and $0<r<r_\mu$. A set $A \subset X$ is
$s$-regular if there is a measure $\mu$ on $X$ which is $s$-regular on $A$
and $\mu(X \setminus A) = 0$. We call a measure $\mu$ on $X$
\emph{doubling} if there are constants $c_\mu \ge 1$ and $r_\mu > 0$
such that
\begin{equation*}
  0 < \mu\bigl( B(x,2r) \bigr) \le c_\mu \mu\bigl( B(x,r) \bigr) <
  \infty
\end{equation*}
for every $x \in X$ and $0 < r < r_\mu$. Clearly, if $\mu$ is
$s$-regular on $X$ then it is also doubling.
Finally, a set $A \subset X$ is called \emph{uniformly (lower)
  $\roo$-porous} provided that there is a constant $r_p > 0$ such that
\begin{equation*}
  \por(A,x,r) > \roo
\end{equation*}
for all $x \in A$ and $0 < r < r_p$, where
\begin{equation*}
  \por(A,x,r) = \sup\{ \roo \ge 0 : B(y,\roo r) \subset B(x,r)
  \setminus A \text{ for some } y \in X \}.
\end{equation*}

It is easy to see that if $\mu$ is a doubling measure on $X$ and $0 <
\alpha < 1$ then
\begin{equation}
  \label{eq:doubling_prop}
  \mu\bigl( B(x,\alpha r) \bigr) \ge c_\mu^{\lfloor \log_2 \alpha
    \rfloor} \mu\bigl( B(x,r) \bigr)
\end{equation}
whenever $x \in X$ and $0 < r < r_\mu$. Here $\lfloor a \rfloor =
\max\{ n \in \Z : n \le a \}$ for $a \in \R$. It is also
straightforward to show that if $A \subset X$ is uniformly
$\roo$-porous for some $\roo > 0$ then $\mu(A) = 0$ for any doubling
measure $\mu$ on $X$. This is true essentially because the points
of $A$ are  nondensity points.

Observe that, by \cite[Theorem 5.7]{Mattila} and \cite[Theorem
3.16]{Cutler}, both the Hausdorff dimension and the packing dimension
of $A$ equal to $s$ provided that $A \subset X$ is $s$-regular.
An unpublished result of Saaranen shows
that if $X$ is $s$-regular then for each $0<t<s$ there exists a
$t$-regular set $A \subset X$. This is the only place where we need
the completeness of $X$. The separability assumption
is natural since the Hausdorff dimension of any nonseparable metric
space is infinity. Also, it can be shown that no $\sigma$-finite
doubling measures exist in nonseparable spaces.
%
%If $A \subset X$ is uniformly $\roo$-porous then $0 \le \roo \le
%1$. In the case $X = \R^n$, we have $0 \le \roo \le \tfrac12$. There
%exist metric spaces having uniformly $1$-porous subsets. 

Our result under consideration is the following. A complete proof with
a slightly better result can be found from \cite{our_paper}.

\begin{theorem}
  Suppose $X$ is $s$-regular. Then $A$ is uniformly $\roo$-porous for
  some $\roo>0$ if and only if there is $t < s$ and a $t$-regular set
  $F \supset A$.
\end{theorem}

\section{Alternative approach}

We begin with the following lemma that gives an alternative proof for
\cite[Corollaries 4.5 and 4.6]{our_paper}. The proof was found while preparing
the article \cite{our_paper} and is also credited by the authors of
that article. Observe, however, that in \cite[Corollary
4.6]{our_paper} we were able to find a better exponent $\delta$. The
lemma generalizes the argument of \cite[Lemma 2.8]{MartioVuorinen}.

Given a set $A \subset X$, we define $A(r) = \{ x \in X :
\dist(x,A) < r \}$ as $r > 0$. With the notation $C(\cdot)$,
we mean a constant whose dependence is indicated in the parentheses.

\begin{lemma} \label{mvlemma}
  Suppose that $\mu$ is a doubling measure on $X$. If $A \subset X$
  is uniformly $\roo$-porous then
  \begin{equation*}
    \mu\Bigl( \bigl( A \cap B(x_0,r_0) \bigr)(r) \Bigr) \le C(c_\mu,\roo)
    \mu\bigl( B(x_0,r_0) \bigr) \Bigl(\frac{r}{r_0}\Bigr)^\delta
  \end{equation*}
  for every $x_0 \in X$ and $0 < r < r_0 \le r_\mu$, where 
  $\delta = C(c_\mu)\roo^{\log_2 c_\mu}/\log(1/\roo)$.
  Moreover, if $\mu$ is $s$-regular then $\delta =
  C(a_\mu,b_\mu,s)\roo^s/\log(1/\roo)$.
\end{lemma}

\begin{proof}
  We may assume that $\roo\leq\tfrac{1}{3}$. Let us denote for
  $0<t_1,t_2<\infty$
  \begin{align*}
    A(t_1,t_2) &= \{ x \in X : t_1 < \dist(x,A) \le t_2 \} \\
    &= A(t_2) \setminus A(t_1).
  \end{align*}
  Furthermore, we denote
  \begin{equation*}
    \alpha_k = \mu\bigl( A(r_0\roo^{3k},r_0\roo^{3(k-1)}) \bigr)
  \end{equation*}
  as $k \in \N$.

  We shall first find a number $0<\gamma<1$
  such that
    $\alpha_k \ge \gamma \sum_{i=k+1}^\infty \alpha_i$
  whenever $k\in\mathbb{N}$. To do so, fix $k$ and notice that for
  every $x \in A$ there is $y \in B(x,r_0 \roo^{3k-2})$ such that
  \begin{equation}
    \label{eq:inc1}
    B(y,r_0 \roo^{3k-1}) \subset B(x,r_0 \roo^{3k-2}) \setminus A.
  \end{equation}
  We claim that this implies %the $r_0 \roo^{3k}$-neighbourhood
  %of $A$ is a subset of the $2r_0\roo^{3k-2}$-neighbourhood of the set
  %$A(r_0 \roo^{3k-1},r_0 \roo^{3k-2})$, that is,
  \begin{equation}
    \label{eq:inc2}
    A(r_0 \roo^{3k}) \subset \bigl( A(r_0 \roo^{3k-1},r_0 \roo^{3k-2})
    \bigr)(2r_0 \roo^{3k-2}).
  \end{equation}
  To see this, pick $z \in A(r_0 \roo^{3k})$. Since $\dist(z,A) \le
  r_0 \roo^{3k}$, there is $x \in A$ such that $d(z,x) <
  r_0 \roo^{3k-1}$. According to \eqref{eq:inc1}, there is $y \in
  A(r_0 \roo^{3k-1},r_0 \roo^{3k-2})$ for which $d(x,y) <
  r_0 \roo^{3k-2}$. Therefore
  \begin{equation*}
    d(z,y) \le d(z,x) + d(x,y) < r_0\roo^{3k-1} + r_0\roo^{3k-2} < 2r_0\roo^{3k-2}
  \end{equation*}
  and hence \eqref{eq:inc2} follows.

  Attaching for each $z \in A(r_0 \roo^{3k-1},r_0 \roo^{3k-2})$ a ball
  $B(z,\tfrac15 r_0 \roo^{3k-2})$, we find, using the $5r$-covering
  theorem (see \cite[Theorem 1.2]{Heinonen} and \cite[Theorem
  2.1]{Mattila}), an index set $I$ and points $z_i \in
  A(r_0 \roo^{3k-1},r_0 \roo^{3k-2})$, $i \in I$, such that
  \begin{equation}
    \label{eq:inc3}
    A(r_0 \roo^{3k-1},r_0 \roo^{3k-2}) \subset \bigcup_{i \in I}
    B(z_i,r_0 \roo^{3k-2})
  \end{equation}
  and $B(z_i,\tfrac15 r_0 \roo^{3k-2}) \cap B(z_j,\tfrac15 r_0 \roo^{3k-2}) =
  \emptyset$ for $i \ne j$.
  Since for each $i \in I$ and
  for every $x \in B(z_i,\tfrac12 r_0 \roo^{3k-1})$ we have $\dist(x,A)
  \le d(x,z_i) + \dist(z_i,A) \le  r_0(\tfrac12
  \roo^{3k-1}+\roo^{3k-2}) \le r_0 \roo^{3(k-1)}$ and
  $\dist(x,A) \ge \dist(z_i,A) - d(z_i,x) > r_0 \roo^{3k-1} -
  \tfrac12 r_0 \roo^{3k-1}
  \ge r_0 \roo^{3k}$, it follows that
  \begin{equation*}
    B(z_i,\tfrac12 r_0\roo^{3k-1}) \subset A(r_0\roo^{3k},r_0\roo^{3(k-1)})
  \end{equation*}
  whenever $i \in I$.
  On the other hand, we have
  \begin{equation*}
    A(r_0 \roo^{3k}) \subset \bigcup_{i \in I} B(z_i,r_0 3\roo^{3k-2})
  \end{equation*}
  by \eqref{eq:inc2} and \eqref{eq:inc3}, and hence, using
  \eqref{eq:doubling_prop},
  \begin{equation}\label{eq:makkara1}
  \begin{split}
    \alpha_k &\ge \sum_{i \in I} \mu\bigl( B(z_i,\tfrac12 r_0 \roo^{3k-1})
    \bigr)
    \geq \gamma\sum_{i \in I} \mu\bigl(
    B(z_i,r_0 3\roo^{3k-2}) \bigr) \\
    &\ge \gamma\mu\bigl( A(r_0 \roo^{3k})
    \bigr) = \gamma\sum_{i=k+1}^\infty
    \alpha_i
  \end{split}
  \end{equation}
  where 
  $\gamma = c_\mu^{\lfloor\log_2(\roo/6)\rfloor}$. 
  Moreover, if $\mu$ is $s$-regular then we have
  $\gamma = \frac{a_\mu}{b_\mu}(\roo/6)^s$. 

  Next we shall derive a growth inequality for the numbers
  $\alpha_k$. Notice that $\alpha_1 \ge \gamma \sum_{i=2}^\infty
  \alpha_i \ge \gamma\alpha_{2}$ by \eqref{eq:makkara1}. Assuming
  inductively
  \begin{equation*}
    \alpha_1 \ge \gamma(\gamma+1)^{k-1} \sum_{i=k+1}^\infty \alpha_i
    \ge \gamma(\gamma+1)^{k-1}\alpha_{k+1}
  \end{equation*}
  for $k\in\mathbb{N}$, we get
  \begin{align*}
    \alpha_1 &\ge \gamma(\gamma+1)^{k-1} \sum_{i=k+1}^\infty \alpha_i =
    \gamma(\gamma+1)^{k-1} \biggl( \alpha_{k+1} +
    \sum_{i=k+2}^\infty \alpha_i \biggr) \\
    &\ge \gamma(\gamma+1)^{k-1} \biggl( \gamma\sum_{i=k+2}^\infty
    \alpha_i + \sum_{i=k+2}^\infty \alpha_i \biggr) \\
    &= \gamma(\gamma+1)^{k} \sum_{i=k+2}^\infty \alpha_i \ge
    \gamma(\gamma+1)^{k} \alpha_{k+2}
  \end{align*}
  using \eqref{eq:makkara1} again. Thus we have shown that
  \begin{equation}
    \label{eq:makkara2}
    \alpha_{k} \le \gamma^{-1}(\gamma+1)^{2-k}\alpha_1
  \end{equation}
  whenever $k\in \N$.

  Now let $0<r<r_0$ and choose $k_0 \in \N$ such that
  $r_0 \roo^{3(k_0+1)} \le r < r_0 \roo^{3k_0}$. Then we have
  \begin{align*}
    \mu\bigl( A(r) \bigr) &\le \mu\bigl( A(r_0 \roo^{3k_0}) \bigr) \le
    \mu(A) + \sum_{i=k_0+1}^\infty \alpha_i \\
    &= \sum_{i=k_0+1}^\infty \alpha_i \le \gamma^{-1} \alpha_{k_0} \le
    \gamma^{-2}(\gamma+1)^{2-k_0} \alpha_{1} \\
    &\le \bigl((\gamma+1)/\gamma\bigr)^2 (\gamma+1)^{-k_0} \mu\bigl(
    A(r_0) \bigr) \\
    &\leq \bigl((\gamma+1)/\gamma\bigr)^2
    \roo^{-k_0 \log(\gamma+1)/\log\roo} \mu\bigl( B(x_0,2 r_0)
    \bigr) \\
    &\le \bigl((\gamma+1)/\gamma\bigr)^2 c_\mu \roo^{-1} \mu\bigl(
    B(x_0,r_0) \bigr) (r/r_0)^{\tfrac{\gamma}{6\log(1/\roo)}}
  \end{align*}
  by using \eqref{eq:makkara1}, \eqref{eq:makkara2}, and the fact that
  $\mu(A) = 0$. Recalling the definition of $\gamma$, the claim follows.
\end{proof}

\section{Main ideas}

To show that a $t$-regular set $A$ is uniformly $\roo$-porous for some
$\roo>0$ whenever $0<t<s$, fix $x \in A$ and small $r>0$ and consider
balls of radius $2^{-k}r$ with given $k \in \N$. It follows from
$s$-regularity that the number of such balls needed to cover $B(x,r)$
is at least a constant times $2^{ks}$. Taking any sufficiently
separated (by means of the $5r$-covering theorem) subcollection from those
balls so that each ball intersects $A$, it follows from $t$-regularity
that the cardinality of such a subcollection is at most a constant
times $2^{kt}$. Choosing $k$ large enough gives the claim.

To sketch the proof of the other direction, let $\delta =
c\roo^s/\log(1/\roo)$ be as in Lemma \ref{mvlemma} and take $s-\delta
< t < s$. We will next construct the set $F \supset A$. Assume for
simplicity that $r_\mu$ and $r_p$ are large. For all $j \in \N$ we use
the $5r$-covering theorem to find a collection of disjoint balls $\bigl\{
B\bigl( x_{ji},(\tfrac12\roo)^j \bigr) \bigr\}_i$ so that $5$ times
bigger balls cover $A$. Using the uniform $\roo$-porosity, choose
$B\bigl( z_{ji},\roo(\tfrac12\roo)^j \bigr) \subset B\bigl(
x_{ji},(\tfrac12\roo)^j \bigr) \setminus A$ for each $i$. Recalling
the result of Saaranen, we construct a $t$-regular set $F_{ji}$ on
each $B\bigl( z_{ji},(\tfrac12\roo)^{j+1} \bigr) =: B_{ji}$ and define $F = A
\cup \bigcup_{i,j} F_{ji}$. It suffices to show that $\nu$, defined to
be the sum of all $t$-regular measures of $F_{ji}$, is $t$-regular on
$F$. We will show that there is a constant $c \ge 1$ such that
$\nu\bigl( B(x,r) \bigr) \le cr^t$ for every $x \in A$ and $r > 0$
small enough. To conclude that $\nu$ is $t$-regular is then
straightforward.

Take $x \in A$, fix a scale $k \in \N$, and denote $N_j = \bigl\{ i :
B_{ji} \cap B\bigl( x,(\tfrac12\roo)^k \bigr) \ne \emptyset \bigr\}$.
It is clear that $B_{ji} \subset A\bigl( (\tfrac12\roo)^j \bigr)$ for
every $j$ and $i$. If $j \le k-1$ then $B_{ji} \cap B\bigl(
x,(\tfrac12\roo)^k \bigr) = \emptyset$ for all $i$. If $j \ge k$ then
$B_{ji} \subset B\bigl( x,4(\tfrac12\roo)^{k+1} \bigr)$ as $i \in
N_j$. Using these observations and Lemma \ref{mvlemma}, we have
\begin{align*}
  c\# N_j \roo^{(j+1)s} &\le \sum_{i \in N_j} \mu(B_{ji}) \le
  \mu\Bigl( \bigl( A \cap B\bigl( x,4(\tfrac12\roo)^{k+1} \bigr)
  \bigr)\bigl( (\tfrac12\roo)^j \bigr) \Bigr) \\
  &\le c\mu\bigl( B\bigl( x,4(\tfrac12\roo)^{k+1} \bigr) \bigr)
  \biggl( \frac{(\tfrac12\roo)^j}{4(\tfrac12\roo)^{k+1}}
  \biggr)^\delta \\
  &\le c\roo^{(s-\delta)(k+1)+j\delta}
\end{align*}
where $\mu$ is the $s$-regular measure,
yielding $\# N_j \le c\roo^{k(s-\delta)-j(s-\delta)}$ for every $j \ge
k$. Here $c$ denotes a constant whose value may vary on each instance
even within a line. This implies, for any $x \in A$ and
$(\tfrac12\roo)^{k+1} \le r < (\tfrac12\roo)^k$,
\begin{align*}
  \nu\bigl( B(x,r) \bigr) &\le \sum_{i,j} \nu\bigl( B_{ji} \cap
  B\bigl( x,(\tfrac12\roo)^k \bigr) \bigr)
  \le c\sum_{j=k}^\infty \# N_j \roo^{(j+1)t} \\
  &\le c\roo^{t+k(s-\delta)}\sum_{j=k}^\infty \roo^{jt-j(s-\delta)} =
  c\roo^{t+k(s-\delta)}\frac{\roo^{kt-k(s-\delta)}}{1-\roo^{t-s+\delta}} \le cr^t
\end{align*}
which is exactly what we wanted.

\end{document}